\documentclass{aupl}  

\usepackage{
amsfonts,
latexsym,
amssymb,
amsmath,
amsthm,
enumerate,
verbatim,
mathrsfs,
epigraph,
}

\usepackage{url}

\newcommand{\labbel}{\label} 

\newtheorem{theorem}{Theorem}
\newtheorem{lemma}[theorem]{Lemma}

\newtheorem{corollary}[theorem]{Corollary}

\newtheorem*{theorem*}{Theorem}
\newtheorem*{corollary*}{Corollary}

\theoremstyle{definition}

\newtheorem{problem}[theorem]{Problem}

\theoremstyle{remark}
\newtheorem{remark}[theorem]{Remark}

 \allowdisplaybreaks[1]

\begin{document}
 
\title{Another characterization of congruence distributive varieties}

\author{Paolo Lipparini} 
\address{Dipartimento Ulteriore di Matematica\\Viale della  Ricerca
 Scientifica\\Universit\`a di Roma ``Tor Vergata'' 
\\I-00133 ROME ITALY}
\urladdr{http://www.mat.uniroma2.it/\textasciitilde lipparin}


\thanks{Work performed under the auspices of G.N.S.A.G.A. Work 
partially supported by PRIN 2012 ``Logica, Modelli e Insiemi''.
The author acknowledges the MIUR Department Project awarded to the
Department of Mathematics, University of Rome Tor Vergata, CUP
E83C18000100006.}

\begin{abstract}
We provide a Maltsev characterization 
of congruence distributive varieties by showing that
a variety $\mathcal {V}$ is congruence distributive 
if and only if the congruence identity 
$\alpha \cap (\beta \circ \gamma \circ \beta )  \subseteq  
\alpha \beta \circ \gamma \circ \alpha \beta \circ \gamma  \dots $
($k$ factors)
holds in $\mathcal {V}$, for some natural number $k$. 
\end{abstract} 

\maketitle  

\centerline {\today}

\smallskip

We assume the reader is familiar with basic notions of lattice theory and
of universal algebra. A small portion of \cite{MMT}
is sufficient as a prerequisite.

A lattice is  distributive if and only if 
it satisfies the identity 
$\alpha (\beta + \gamma )  \leq \alpha \beta + \gamma $.
It follows that an algebra $\mathbf A$ 
is congruence distributive  
if and only if, for all congruences
$\alpha$, $\beta$ and $\gamma$ of $\mathbf A$ 
and for every $h$,  
the inclusion 
$\alpha (\beta \circ_h \gamma )  \subseteq  \alpha \beta + \gamma $.
 holds.
Here juxtaposition denotes intersection,
$+$ is join in the congruence lattice and
$\beta \circ_h \gamma$ is
$\beta \circ \gamma \circ \beta \circ \gamma \dots $
with $h$ factors ($h-1$ occurrences of  $ \circ $).
 
Considering now a variety $\mathcal {V}$,
it follows from standard arguments in the theory
of Maltsev conditions that $\mathcal {V}$ is congruence distributive 
if and only if, for every $h$, there is some $k$
such that the congruence identity   
\begin{equation}\labbel{hk}      
\alpha (\beta \circ_h \gamma )  \subseteq  \alpha \beta \circ_k \gamma 
 \end{equation}
holds in $\mathcal {V}$.
The naive  expectation 
(of course, motivated by \cite{JD})
that 
the congruence identity 
\begin{equation}\labbel{kk}     
\alpha (\beta \circ \gamma )  \subseteq  \alpha \beta \circ_k \gamma ,
\quad \text{ for some $k$},
  \end{equation}
 is enough to imply congruence distributivity is false. 
Indeed, by \cite[Theorem 9.11]{HMK}, 
a locally finite variety $\mathcal {V}$ satisfies \eqref{kk}
if and only if  $\mathcal {V}$ omits types
{\bf 1, 2, 5}.
More generally, with no finiteness assumption,
Kearnes and Kiss \cite[Theorem 8.14]{KK} 
proved that a variety $\mathcal {V}$ 
satisfies  \eqref{kk} if and only if $\mathcal {V}$ is 
join congruence semidistributive. Many other interesting equivalent
conditions are presented in \cite{HMK,KK}.

In spite of the above results, we show that the next step is enough, namely,
if we take $h=3$ in identity \eqref{hk},
we get a condition implying congruence distributivity.
After a short elementary proof
relying on \cite{D,JD},
in Remark \ref{ct} we sketch an alternative argument
which relies only on \cite{contol}.  
Then, by working directly with the terms associated to the
Maltsev condition arising from 
\eqref{hk} for $h=3$, we show that this instance
of  \eqref{hk}
implies
$\alpha (\beta \circ \gamma \circ \beta )  \subseteq
\alpha \beta \circ_r \alpha \gamma  $,
for some $r < \frac{k^2}{2} $.  

\begin{theorem} \labbel{piccolo}
A variety $\mathcal {V}$ is congruence distributive 
if and only if the identity 
\begin{equation}\labbel{1} 
\alpha (\beta \circ \gamma \circ \beta ) \subseteq 
\alpha \beta +  \gamma  
  \end{equation}    
holds in every congruence lattice of algebras in  $\mathcal {V}$.
 \end{theorem}

 \begin{proof}
If $\mathcal {V}$ is congruence distributive, then
$\alpha (\beta \circ \gamma \circ \beta ) \subseteq 
\alpha ( \beta + \gamma ) \leq 
\alpha \beta +  \gamma  $.

For the nontrivial direction, 
assume that \eqref{1} holds in $\mathcal {V}$.
By taking $\alpha \gamma $ in place of $\gamma$ 
in \eqref{1} we get 
$\alpha (\beta \circ \alpha \gamma \circ \beta ) \subseteq 
\alpha \beta +  \alpha \gamma  $.
 Day \cite{D} has showed that this identity implies congruence modularity
within a variety.
From \eqref{1} and congruence modularity we get
$\alpha (\beta \circ \gamma \circ \beta ) \subseteq 
\alpha (\alpha \beta +  \gamma ) = \alpha \beta + \alpha \gamma  $
and, since trivially
$\alpha (\beta \circ \gamma ) \subseteq \alpha (\beta \circ \gamma \circ \beta ) $,
we obtain
$\alpha (\beta \circ \gamma )  \subseteq \alpha \beta + \alpha \gamma$. 
Within a variety this identity implies congruence distributivity
by \cite{JD}.
 \end{proof}  

It is standard to express Theorem \ref{piccolo}
in terms of a  Maltsev condition.

\begin{corollary} \labbel{main}
A variety 
$\mathcal {V}$ is congruence distributive 
if and only if there is some $k$
such that  any one of the following equivalent conditions
hold.
\begin{enumerate}[(i)]   
\item 
$\mathcal {V}$ satisfies the congruence identity   
\begin{equation}\labbel{hk3} 
\alpha (\beta \circ \gamma \circ \beta )  \subseteq  \alpha \beta \circ_k \gamma .
 \end{equation}
\item
The identity \eqref{hk3} holds in 
$\mathbf F _{ \mathcal V } ( 4 ) $, the free
algebra in $\mathcal {V}$ generated by 
four elements $x,y,z,w$;
actually, it is equivalent to assume that     
\eqref{hk3} holds in 
$\mathbf F _{ \mathcal V } ( 4 ) $
in the special case when
when $\alpha= Cg(x,w)$,
$\beta=Cg((x,y),(z,w))$ and
$\gamma= Cg(y,z)$.
\item
$\mathcal V$ has $4$-ary  terms 
$d_0, \dots, d_{k}$ such that the following 
equations are valid in $\mathcal V$:
  \begin{enumerate}    
\item 
$x=d_0(x,y,z,w)$;
\item
$d_i(x,x,w,w)=d_{i+1}(x,x,w,w)$, for $i$ even;
\item
$d_i(x,y,z,x)=d_{i+1}(x,y,z,x)$, for $i$ even;
\item
$d_i(x,y,y,w)=d_{i+1}(x,y,y,w)$, for $i$ odd, and
\item 
$d_{k}(x,y,z,w)=w$.
  \end{enumerate}  
  \end{enumerate} 
 \end{corollary} 

\begin{proof} 
(i) $\Rightarrow $  (ii) is trivial;
(ii) $\Rightarrow $  (iii) 
and (iii) $\Rightarrow $  (i)
 are standard; for example,
there is no substantial difference with respect to \cite{D}.
See, e.~g., \cite{G,CV,ntcm} for further details, or \cite{Pal,W}
for a more general form of the arguments.
Thus we have that (i) - (iii) are equivalent, for any given  $k$.

Clearly congruence distributivity implies 
the second statement in (ii), for some $k$;
moreover identity \eqref{hk3} in (i)
implies identity \eqref{1}, hence congruence distributivity
follows from Theorem \ref{piccolo}. 
\end{proof}

\begin{remark} \labbel{ct}  
It is possible to give a direct proof
that clause (i) in Corollary  \ref{main}
implies congruence distributivity by using a theorem 
from \cite{contol} and without resorting to 
\cite{D,JD}.
By \cite[Theorem 3 (i) $\Rightarrow $  (iii)]{contol},
a variety $\mathcal {V}$ satisfies identity \eqref{hk3}
for congruences if and only if $\mathcal {V}$ satisfies the same identities
when $\alpha$, $\beta$ and $\gamma$ are \emph{representable tolerances}.
A tolerance $\Theta$ is \emph{representable}
if it can be expressed as $\Theta = R \circ R ^\smallsmile $,
for some admissible relation $R$,  where $R ^\smallsmile $ 
denotes the \emph{converse} of $R$.
 To show congruence distributivity, notice that 
 the relation
$\Delta_m = \beta   \circ_m  \gamma   $
is a representable tolerance, for every odd $m$.   
By induction on $m$,
it is easy to see that 
the identity \eqref{hk3}, when interpreted
for representable tolerances,
implies 
$\alpha (\Delta_m \circ \gamma   \circ \Delta_m )
  \subseteq  \alpha \beta \circ _p  \gamma $, for every 
odd $m$ and some appropriate $p$ depending on $m$. 
In particular, we get that, for every $h$, there is some $p$ such that  
\begin{equation}\labbel{hp}     
 \alpha ( \beta \circ_h \gamma ) \subseteq \alpha \beta \circ _p  \gamma ,
 \end{equation}
hence also
$ \alpha ( \beta \circ_h \gamma ) \subseteq 
\alpha (\alpha \beta \circ _p  \gamma ) 
= \alpha \beta \circ \alpha ( \gamma \circ _{p-1} \alpha \beta ) $.
 Taking now 
$\gamma$ in place of $\beta$,
$\alpha \beta $ in place of $\gamma$ 
and $p-1$ in place of  $h$ in \eqref{hp},
we get 
$\alpha ( \gamma \circ _{p-1} \alpha \beta ) \subseteq 
\alpha \gamma \circ _q \alpha \beta $, for some $q$,
thus $\alpha ( \beta \circ_h \gamma ) \subseteq
 \alpha \beta \circ _{q+1}  \alpha \gamma$.
In particular, 
 $\alpha ( \beta \circ_h \gamma ) \subseteq
 \alpha \beta +  \alpha \gamma$, for every $h$,
hence we get congruence distributivity, by a remark at the beginning. 
 Compare \cite{ntcm} for corresponding arguments. 
If one works out the details, one obtains that
if $k \leq 2^t$, $t \geq 1$ and $\ell \geq 2$,
then identity \eqref{hk3} implies
$\alpha (\beta  \circ _{2^\ell -1}  \gamma ) 
 \subseteq 
 \alpha \beta \circ _{2^{s}+1} \alpha  \gamma$, 
with $s=(t-1)^2( \ell -1) +1 $,
a rather large number of factors on the right.
We shall present explicit details in the Appendix.

We are now going to show that we can obtain a lighter 
bound on the right using different methods.
 \end{remark} 

\begin{remark} \labbel{mod}
Notice that if some sequence of terms satisfies Clause (iii) in
Corollary \ref{main},
then the terms  satisfy also
  \begin{enumerate}  
  \item[(f)]
$x=d_i(x,y,y,x)$, for every $i \leq k$.  
   \end{enumerate} 
This follows immediately by induction from (a), (c) and (d).
From the point of view of congruence identities,
this corresponds to taking $ \alpha \gamma $
in place of $\gamma$ in \eqref{1}, as we did
in the proof of Theorem \ref{piccolo}.
At the level of Maltsev conditions,
this gives a proof that Clause (iii) in Corollary  \ref{main}     
implies congruence modularity,
since the argument shows that the terms
$d_0, \dots, d_k$ obey Day's conditions \cite{D} 
for congruence modularity. 
\end{remark}

\begin{theorem} \labbel{level} 
If some variety $\mathcal {V}$ satisfies the congruence identity
\eqref{hk3} $\alpha (\beta \circ \gamma \circ \beta )  \subseteq  \alpha \beta \circ_k \gamma$, for some $k \geq 3$, then $\mathcal {V}$ satisfies
\begin{equation*}
\alpha (\beta \circ \gamma \circ \beta )  \subseteq  \alpha \beta \circ_r \alpha \gamma,
 \end{equation*}      
where $r=  \frac{k^2-4k+9}{2}$ for $k$ odd, and
$r=  \frac{k^2-3k+4}{2}$ for $k$ even.
\end{theorem}

 \begin{proof}
By Corollary   \ref{main},
we have terms as given by (iii).
Suppose that 
$(a,d) \in \alpha ( \beta   \circ  \gamma  \circ \beta )$
in some algebra in $\mathcal {V}$.
Thus 
$a \mathrel { \alpha } d $
and
$ a \mathrel   \beta   b \mathrel   \gamma  c  \mathrel   \beta  d$,
for certain elements $b $ and $ c$.
We claim that 
\begin{equation}\labbel{gabg}    
(d_i(a,b,b,d),d _{i+2}(a,b,b,d))
\in  \alpha ( \gamma \circ \alpha \beta \circ \gamma ),
   \end{equation}
for every odd index $i < k-1$.
Indeed,   
\begin{equation*}
  d_i(a,b,b,d) \mathrel \alpha  
d_i(a,b,b,a) = a =d _{i+2}(a,b,b,a)
\mathrel \alpha  d _{i+2}(a,b,b,d) ,
 \end{equation*}    
by (f) in the above remark.
Moreover, still assuming $i$ odd,  
\begin{gather*}    
   d_{i+1}(a,b,c,d)  \mathrel { \beta }
d_{i+1}(a,a,d,d)=
d_{i+2}(a,a,d,d) \mathrel \beta 
 d _{i+2}(a,b,c,d), \text{  and}
\\
 d_{i+1}(a,b,c,d)  \mathrel { \alpha  }
d_{i+1}(a,b,c,a)=
d_{i+2}(a,b,c,a) \mathrel \alpha  
 d _{i+2}(a,b,c,d), \text{  hence  }  
\\
  d_i(a,b,b,d) {{\hspace {2pt}}{=}{\hspace {2pt}}} 
 d_{i+1}(a,b,b,d)  \mathrel \gamma 
d_{i+1}(a,b,c,d)  \mathrel {\alpha \beta } d _{i+2}(a,b,c,d)
\mathrel \gamma   d _{i+2}(a,b,b,d) ,
\end{gather*} 
thus \eqref{gabg} follows. 
From \eqref{gabg} and 
\eqref{hk3} with $\gamma$ in place of $\beta$
and $\alpha \beta $ in place of $\gamma$, 
we get
\begin{equation*}\labbel{gabg'}    
(d_i(a,b,b,d),d _{i+2}(a,b,b,d))
\in  \alpha  \gamma \circ_k \alpha \beta,
   \end{equation*}
for every odd index $i$.

Arguing as above, 
$a \mathrel {\alpha \beta } d_1(a,b,c,d) \mathrel { \alpha \gamma }  d_1(a,b,b,d)  $.
If $k$ is odd, then 
\begin{equation*}      
d_{k-2}(a,b,b,d) = d_{k-1}(a,b,b,d) 
\mathrel { \alpha \beta } d_{k-1}(a,a,d,d) =
 d_{k}(a,a,d,d) = d,
\end{equation*}
 thus 
the elements
$d_1(a,b,b,d), d_3(a,b,b,d), \dots, d_{k-2}(a,b,b,d)$ witness
\begin{equation} \labbel{bago}    
(a,d) \in \alpha \beta \circ \alpha \gamma \circ 
( \alpha  \gamma \circ_k \alpha \beta ) ^{ \frac{k-3}{2}  } \circ \alpha \beta  =
\alpha  \beta  \circ_{r } \alpha \gamma ,
  \end{equation}
  for $r = \frac{k^2-4k+9}{2} $.
In the computation of $r$ 
we have used that, say,
$( \alpha  \gamma \circ_k \alpha \beta ) \circ ( \alpha  \gamma \circ_k \alpha \beta )
=  \alpha  \gamma \circ_{2k-1} \alpha \beta $,
$( \alpha  \gamma \circ_k \alpha \beta )^3 =
 \alpha  \gamma \circ_{3k-2} \alpha \beta$, etc.,
since $k$ is odd, hence there are adjacent occurrences
of $\alpha \gamma $ which  join into one.
In the general case, 
$( \alpha  \gamma \circ_k \alpha \beta )^t =
 \alpha  \gamma \circ_{t(k-1) + 1} \alpha \beta$,
for $k$ odd.
Finally,
we have two adjacent   occurrences of 
$\alpha \gamma $  at the second and third
 place in \eqref{bago}, too. From the above
observations we get the value $ \frac{k^2-4k+9}{2} $ of
$r$.

On the other hand, if $k$ is even, then 
$d_{k-1}(a,b,b,d) = d_k(a,b,b,d) =d $.
Moreover, 
since $d_1(a,b,c,d) \mathrel { \alpha \gamma }  d_1(a,b,b,d)$,
we have by \eqref{gabg} 
\begin{equation} \labbel{blu}
    (d_1(a,b,c,d) , d_3(a,b,b,d)) \in
\alpha \gamma \circ  \alpha ( \gamma \circ  \alpha \beta \circ \gamma )  =
 \alpha ( \gamma \circ  \alpha \beta \circ \gamma )   ,
 \end{equation} 
hence we can consider 
$d_1(a,b,c,d)$ in place of  $ d_1(a,b,b,d)$.
By considering the converse of 
\eqref{hk3},  we get
\begin{equation}\labbel{baba}    
\alpha ( \beta \circ \gamma \circ \beta ) \subseteq 
\gamma \circ_k \alpha \beta ,    \text{  if $k$ is even.} 
 \end{equation}  
Taking $\gamma$ in place of $\beta$ 
and $\alpha \beta $ in place of $\gamma$ 
in \eqref{baba}, then from \eqref{blu} we get
\begin{equation*}
(d_1(a,b,c,d) , d_3(a,b,b,d))
\in  \alpha  \beta  \circ_k \alpha \gamma .
 \end{equation*}    
We can go on the same way, using alternatively
\eqref{baba} and \eqref{hk3}
and considering the elements
$d_1(a,b,c,d) , d_3(a,b,b,d), d_5(a,b,b,d), \dots, 
d_{k-3}(a,b,b,d)$, 
getting 
  $(a,d) \in 
( \alpha  \beta  \circ_k \alpha \gamma  ) 
\circ_{ \frac{k-2}{2} }
( \alpha  \gamma   \circ_k \alpha \beta )  =
\alpha  \beta  \circ_{ r} \alpha \gamma $, for $r = \frac{k^2-3k+4}{2}  $.
 \end{proof}

We expect that the evaluation of $r$ in Theorem \ref{level}
can be further improved, but we have no guess as to what extent.  

One can consider an identity intermediate between
\eqref{kk} and  \eqref{hk} 
by shifting the occurrence of $\alpha$ the other way,
with respect to \eqref{hk3}.

\begin{problem} \labbel{shift}
Within a variety, is the following identity equivalent to congruence 
distributivity?  
\begin{equation}\labbel{hk3sh} 
\alpha (\beta \circ \gamma \circ \beta )  \subseteq   \beta + \alpha \gamma 
 \end{equation}
 \end{problem}

We are not claiming that the above problem is difficult;
in any  case, it is not solved by the present note.
As usual, a variety satisfies \eqref{hk3sh}
if and only if there is some $k$ such that   
$\alpha (\beta \circ \gamma \circ \beta )  \subseteq   \beta \circ_k \alpha \gamma $
holds in $\mathcal V$.
Let us  also notice that the identity 
\eqref{hk3sh} implies congruence distributivity
if and only if it implies congruence modularity.
Indeed, if 
\eqref{hk3sh} implies congruence modularity,
then we get distributivity arguing as in the last two sentences of the proof
of Theorem \ref{piccolo}.

\smallskip 

{\scriptsize The author considers that it is highly  inappropriate, 
 and strongly discourages, the use of indicators extracted from the list below
  (even in aggregate forms in combination with similar lists)
  in decisions about individuals (job opportunities, career progressions etc.), 
 attributions of funds  and selections or evaluations of research projects. \par }

\section*{Appendix} \labbel{Appendix} 

In this appendix we justify the values reported in 
Remark \ref{ct}.

\begin{lemma} \labbel{lemap}
The conditions in Corollary \ref{main}
are also equivalent to: 

(iv) For every algebra 
$\mathbf A \in \mathcal V$, the following identity 
\begin{equation}\labbel{nte}    
\alpha ( \Delta  \circ  \gamma  \circ \Delta) \subseteq 
 \alpha \Delta  \circ_ \kappa \gamma,  
  \end{equation}
holds, for all congruences $\alpha$ and $ \gamma $ 
on $\mathbf A$
and every tolerance  $\Delta$  on $\mathbf A$ such that  there exists
an admissible relation $R$ on $\mathbf A$ for which 
$\Delta = R \circ R^\smallsmile$.
 \end{lemma} 

\begin{proof} 
The equivalence of (i) and (iv)
is a special case of
\cite[Theorem 3 (i) $\Rightarrow $  (iii)]{contol}.
For the reader's convenience, we sketch a direct proof of
(iii) $\Rightarrow $  (iv), while, of course,
(iv) $\Rightarrow $  (i) is obvious.

So let us assume that we have terms as given by
(iii) and that $\alpha$, $\Delta$ and $\gamma$ satisfy the assumptions
in (iv). Suppose that 
$(a,d) \in \alpha ( \Delta  \circ  \gamma  \circ \Delta)$,
thus 
$a \mathrel { \alpha } d $
and
$ a \mathrel   \Delta  b \mathrel   \gamma  c  \mathrel   \Delta d$,
for certain $b,c \in A$.
Moreover, by the assumption on $\Delta$,
$a \mathrel { R} b' \mathrel { R ^\smallsmile } b  $
and
$c \mathrel { R} c' \mathrel { R ^\smallsmile } d  $,
for certain $b', c' \in A$.
We claim that the elements
  $d_i(a,b,c,d) $,
for $i=0, \dots, k$,  
witness
that 
$(a,d) \in \alpha \Delta  \circ_ \kappa \gamma$.
For example, let us check that 
$d_i(a,b,c,d)  \mathrel \Delta d_{i+1}(a,b,c,d) $,
for $i$ even. Indeed,
 $d_i(a,b,c,d)  \mathrel R 
d_i(b',b',c',c') =
d_{i+1}(b',b',c',c') \mathrel { R ^\smallsmile } 
 d_{i+1}(a,b,c,d) $,
since, say, $b \mathrel R b'$
and since, by assumption,
$\Delta = R \circ R^\smallsmile$.
 All the rest is standard and simpler.
We have proved that (i) - (iv) are equivalent, for every $k$.
\end{proof}    

We now prove that (iv) implies
$\alpha (\beta + \gamma )  \leq \alpha \beta + \gamma $,
an identity equivalent to distributivity.
We shall actually  show that 
if $k$ is even, say, $k=2r$, then  
(iv) implies
\begin{equation}\labbel{bip} 
\alpha (\beta  \circ _{2^\ell -1}  \gamma ) 
 \subseteq  \alpha \beta \circ _{2r^{\ell-1}} \gamma ,
\quad
\text{ for every $\ell \geq 2$. } 
  \end{equation}   
Clearly, it is no loss of generality to assume that 
$k$ is even, since if the identity \eqref{nte} holds for some
odd $k$, then   \eqref{nte} holds for $k+1$, as well.
Moreover, if 
$(a,b) \in \alpha (\beta + \gamma )$ 
in some algebra, then   
$(a,b) \in  \alpha (\beta  \circ _{2^\ell -1}  \gamma ) $,
for some sufficiently large $\ell$ depending on $a$ and $b$.
Hence, in order to show
congruence distributivity,  it is enough to prove the identity \eqref{bip}. 

The proof of \eqref{bip} is by induction on 
$\ell \geq 2$. 
The base case $\ell= 2$ is the special case 
$\Delta = \beta $ of identity \eqref{nte}.  
Suppose that   the identity \eqref{bip} holds for some
$\ell \geq 2$ and 
set $\Delta= \beta  \circ _{2^{\ell} -1}  \gamma$.
By the inductive hypothesis, we have
$\alpha \Delta  \subseteq  \alpha \beta \circ _{2r^{\ell-1}} \gamma$. 

If $R= \beta  \circ _{2^{\ell-1}}  \gamma$,
 then 
$\Delta = R \circ R ^\smallsmile $.
Indeed, $\ell \geq 2$, thus 
$2^{\ell-1}$ is even, hence the last factor in the definition
of $R$ is $\gamma$ and $\gamma$ is also the first factor
of $R ^\smallsmile  $. 
Since $\gamma$  is a congruence, we have
$\gamma \circ \gamma = \gamma $, namely, one factor absorbs in 
$R \circ R ^\smallsmile $,
thus $R \circ R ^\smallsmile $ has 
$2^{\ell} -1$ factors, 
hence
$\Delta = R \circ R ^\smallsmile $.
Thus we can apply (iv) and we have
\begin{equation*} 
\alpha (\beta  \circ _{2^{\ell+1} -1}  \gamma )
= \alpha (\Delta \circ \gamma \circ \Delta  )
\subseteq ^{\eqref{nte}}  
\alpha \Delta \circ _{2r} \gamma 
\subseteq  ^{\rm ih}
(\alpha \beta \circ _{2r^{\ell-1}} \gamma) \circ _{2r} \gamma 
= \alpha \beta \circ_{2r ^ \ell } \gamma ,  
\end{equation*}  
where the superscripts \eqref{nte}
and ``ih''  mean that we have applied, respectively, identity 
\eqref{nte} and the inductive hypothesis
and where in the last identity we have used 
again $\gamma \circ \gamma = \gamma $,
noticing that   $2r^{\ell-1}$ is even, hence the last factor
in the expression
$\alpha \beta \circ _{2r^{\ell-1}} \gamma$ is $\gamma$.

The induction step is thus complete, hence we have proved \eqref{bip}.

In the next corollary we state explicitly some informations
which can be obtained from the above arguments. 

\begin{corollary} \labbel{cor}
If some variety $\mathcal {V}$ satisfies one of the equivalent 
conditions in Theorem \ref{main} with $k \leq 2r$,
then  $\mathcal {V}$ satisfies the identity \eqref{bip},
for every $\ell \geq 2$.

If in addition
$k \leq 2^p$, for some $p \geq 1$, then,
for every $\ell \geq 2$,  $\mathcal {V}$ satisfies
\begin{equation*}
\alpha (\beta  \circ _{2^\ell -1}  \gamma ) 
 \subseteq 
 \alpha \beta \circ _{2^{s}+1} \alpha  \gamma, 
 \end{equation*} 
where $s=(p-1)^2( \ell -1) +1 $. In particular, taking $\ell=2$,
we get that $\mathcal {V}$ satisfies  
\begin{equation*}
\alpha (\beta  \circ   \gamma \circ \beta ) 
 \subseteq 
 \alpha \beta \circ _{2^{t}+1} \alpha  \gamma,
 \end{equation*} 
for $t=(p-1)^2 +1$. 
 \end{corollary}

 \begin{proof} 
The first statement is given by the above proof 
of \eqref{bip}.

To prove the second statement, 
it is no loss of generality to assume
that $k=2r$ and $k = 2^p$.
Notice that from \eqref{bip} we get   
\begin{equation}\labbel{bipbip} 
\alpha (\beta  \circ _{2^\ell -1}  \gamma ) 
 \subseteq 
 \alpha ( \alpha \beta \circ _{2r^{\ell-1}} \gamma )
=
\alpha \beta \circ \alpha ( \gamma  \circ _{2r^{\ell-1}-1} \alpha \beta ),
  \end{equation}   
since $\alpha$ is supposed to be a congruence, in particular, transitive.
From $k=2r$ and  $k=2^p$, we get 
$r=2 ^{p-1} $, hence 
  $2r^{\ell-1}-1 = 2^q -1$,
for $q= (p-1)( \ell -1) +1$. 
Applying \eqref{bip}
with $\gamma$ in place of $\beta$,
with $\alpha \beta $ in place of $\gamma$ 
and $q$ in place of $\ell$,
we get 
\begin{equation*}
\alpha ( \gamma  \circ _{2r^{\ell-1}-1} \alpha \beta ) 
=
\alpha ( \gamma  \circ _{2^{q}-1} \alpha \beta )
\subseteq 
\alpha \gamma  \circ _{2r^{q-1}} \alpha \beta 
 \end{equation*}    
and the conclusion follows from
\eqref{bipbip}, since 
$2r^{q-1} = 2 (2 ^{p-1})^{(p-1)( \ell -1)} =2^s $.
\end{proof}

\end{document}